\documentclass{amsart}

\usepackage{amssymb}
\usepackage{amscd}
\usepackage[all,cmtip]{xy}
\usepackage{hyperref}
\usepackage{graphicx}

\newcommand{\comment}[1]{}

\newcommand{\Ps}{\mathbf{P}} 
\newcommand{\Q}{\mathbf{Q}}
\newcommand{\R}{\mathbf{R}}

\newcommand{\m}{\mathfrak{m}}

\newcommand{\kv}{\mathrm{k}}

\newcommand{\Frob}{\mathrm{Frob}}

\DeclareMathOperator{\im}{im}

\DeclareMathOperator{\Gal}{Gal}

\DeclareMathOperator{\Hom}{Hom}

\DeclareMathOperator{\Aut}{Aut}

\DeclareMathOperator{\Dv}{D}
\DeclareMathOperator{\Iv}{I}

\DeclareMathOperator{\Graph}{Graph}

\theoremstyle{plain}
\newtheorem{theorem}{Theorem}[section]

\newtheorem{corollary}[theorem]{Corollary}

\newtheorem{lemma}[theorem]{Lemma}

\newtheorem{proposition}[theorem]{Proposition}

\theoremstyle{definition}
\newtheorem{definition}[theorem]{Definition}

\newtheorem{problem}{Problem}

\begin{document}

\title{Images of polynomial maps on large fields}
\author{Michiel Kosters}
\address{Mathematisch Instituut
P.O. Box 9512
2300 RA Leiden
the Netherlands}
\email{mkosters@math.leidenuniv.nl}
\urladdr{www.math.leidenuniv.nl/~mkosters}
\date{\today}
\thanks{This article is part of my PhD thesis written under the supervision of Hendrik Lenstra.}

\begin{abstract}
A field $k$ is called large if every irreducible $k$-curve with a $k$-rational smooth point has infinitely many $k$-points. 
Let $k$ be a perfect large field and let $f \in k[x]$. Consider the evaluation map $f_k: k \to k$. Assume that $f_k$ is not surjective. We will show that $k \setminus f_k(k)$ is infinite. This conclusion follows from a similar statement about finite morphisms between normal projective curves over perfect large fields.
\end{abstract}

\maketitle

\section{Introduction}

\begin{problem}
Let $k$ be an infinite field and let $f \in k[x]$. Consider the evaluation map $f_k: k \to k$. Assume that $f_k$ is not surjective. Is $k \setminus f_k(k)$ infinite?
\end{problem}

This question was asked by Philipp Lampe on mathoverflow as
\href{http://mathoverflow.net/questions/6820/can-a-non-surjective-polynomial-map-from-an-infinite-field-to-itself-miss-only-f/7244#7244}{Question
6820} in 2009. The problem is still unsolved in general. It is not hard to see that the problem has a positive answer when $k$ is Hilbertian. Number fields are for example Hilbertian. In this article we will positively answer this problem for perfect large fields. 
Let us first define what a large field is.

\begin{definition}
A field $k$ is called \emph{large}\index{large field} if every irreducible $k$-curve $C$ with a $k$-rational smooth point has infinitely many $k$-points.
\end{definition}

Note that large fields are infinite. For more information on large fields see the survey \cite{POP}. Some examples of large fields are $\R$, $\Q_p$ ($p$ prime), $l((t))$ (where $l$ is a field), infinite algebraic extensions of finite fields. Furthermore, finite extensions of large fields are large. In the definition of large field we can replace the word smooth by normal if $k$ is perfect (\cite[Corollary 4.3.33 and
Lemma 8.2.21]{LIU})

Our main theorem is the following (we follow \cite{LIU} for the definition of a normal projective curve).

\begin{theorem} \label{25c12}
 Let $k$ be a perfect large field. Let $C, D$ be normal projective curves over $k$. Let $f \colon C \to D$ be a finite morphism. Suppose that the induced map
$f_k \colon C(k) \to D(k)$ is not surjective. Then one has $|D(k) \setminus f_k(C(k))|=|k|$.
\end{theorem}

The following corollary gives a positive answer to our problem for perfect large fields, since large fields are infinite.

\begin{corollary} \label{25c14}
Let $k$ be a perfect large field.
Then the following hold.
\begin{enumerate}
 \item Let $f \in k(x)$ such that the induced map $f_k \colon \Ps^1_k(k) \to \Ps^1_k(k)$ is not surjective. Then one has $|\Ps^1_k(k) \setminus f_k(\Ps^1_k(k))|=|k|$.
 \item Let $f \in k[x]$ such that the induced evaluation map $f_k \colon k \to k$ is not surjective. Then one has $|k \setminus f_k(k)|=|k|$.
\end{enumerate}
\end{corollary}

\section{Prerequisites}

\subsection{Valuations and curves}

\begin{definition}
Let $K$ be a field. Then a \emph{valuation ring} on $K$ is a subring $\mathcal{O}$ of $K$ such that for all $x \in K^*$ one has $x \in \mathcal{O}$ or $x^{-1} \in \mathcal{O}$. 
\end{definition}

Such a valuation ring is a local ring and it is integrally closed in $K$. For notational convenience, we often say that $P$ is a valuation with corresponding valuation ring $\mathcal{O}_P$, maximal ideal $\m_P$ and residue field $\kv_P$. We say that $(K,P)$ is a valued field.

Let $(K,P)$ be a valued field. Let $L$ be a field extension of $K$. A valuation $Q$ on $L$ is an extension of $P$ if $\mathcal{O}_Q \cap K = \mathcal{O}_P$. We denote such an extension by $(K,P) \subseteq (L,Q)$. Such an extension is called finite (Galois, normal, \ldots) if $L/K$ is finite (Galois, normal, \ldots).

\begin{definition}
Let $k$ be a field. A \emph{function field} over $k$ is a finitely generated field extension of $k$ of transcendence degree $1$.
\end{definition}

\begin{theorem} \label{2c411}
There is an anti-equivalence of categories between the category of normal projective curves over $k$ with finite morphisms
and the category of function fields over $k$ with finite $k$-morphisms of fields. This equivalence maps a curve $C$ to its function field $k(C)$ and a finite morphism $C \to D$ to the induced inclusion $k(D) \subseteq k(C)$. 
\end{theorem}
\begin{proof}
 See \cite[Proposition 7.3.13]{LIU}. 
\end{proof}

Let $C$ be a normal projective curve over a field $k$. Each closed point $P$ gives rise to a valuation on $k(C)$ which contains $k$ but is not equal to $k(C)$, namely $\mathcal{O}_P$. Conversely, every such valuation corresponds to a point. This set of such valuations is denoted by $\mathcal{P}_{K/k}$. We let $\mathcal{P}^1_{K/k} \subseteq \mathcal{P}_{K/k}$ be the subset consisting of valuations $P$ which satisfy $\kv_P=k$. When $K=k(C)$, the set $\mathcal{P}_{K/k}^1$ corresponds to $C(k)$.

\subsection{Galois theory of valuations}

We refer to \cite{KO7} for a Galois theoretic approach to valuation theory. We recall some of the important statements.

Let $K$ be a field with valuation $P$ and let $M$ be a normal extension of $K$ with group $G=\Aut_K(M)$. Let $Q$ be an extension of $P$ to $M$.
Set $\Dv_{Q,K}=\{g \in G: gQ=Q\}$ (\emph{decomposition group}). Note that we have a natural map $\Dv_{Q,K} \to \Aut_{\kv_P}(\kv_Q)$. The kernel of this map is called the \emph{inertia group} and is denoted by $\Iv_{Q,K}$. 

The following theorem summarizes the results from general valuation theory which we need later on.

\begin{theorem} \label{984}
Let $L$ be an intermediate extension of $M/K$. The following hold:
\begin{enumerate}
\item The set of valuations extending $P$ to $M$ is not empty and the group $G$ acts transitively on it.
\item The sequence 
\begin{eqnarray*}
0 \to \Iv_{Q,K} \to \Dv_{Q,K} \to \Aut_{\kv_P}(\kv_Q) \to 0
\end{eqnarray*}
is exact. It is right split (as topological groups) if $\kv_Q$ is algebraically closed. 
\item One has $\Dv_{Q,L} = \Aut_L(M) \cap \Dv_{Q,K}$ and $\Iv_{Q,L} = \Aut_L(M) \cap \Iv_{Q,K}$.
\item If $L/K$ is Galois, then then the maps $\Dv_{Q,K} \to \Dv_{Q|_L,K}$ and $\Iv_{Q,K} \to \Iv_{Q|_L,K}$ are surjective. 
\end{enumerate}
\end{theorem}
\begin{proof}
i: See \cite[Proposition 5.6]{KO7}.

ii: See \cite[Theorem 3.8iii]{KO7}.

iii: Follows directly from the definitions.

iv: See \cite[Theorem 3.6]{KO7}.
\end{proof}

One of the reasons for assuming that $k$ is perfect in our main theorem in the following proposition.

\begin{proposition} \label{1c388}
Let $(K,P)$ be a valued field and let $L$ be a finite algebraic extension of $K$. Assume that $\kv_P$ is perfect. Let $(M,Q) \supseteq (K,P)$ be a finite normal extension of valued fields with group $G=\Aut_K(M)$ such that the $G$-set $X=\Hom_K(L,M)$ is not empty. Then the cardinality of the set of valuations $P'$ on
$L$ extending $P$ such that $\kv_{P'}=\kv_P$ is equal to $\# \left(\Iv_{Q,K} \backslash X \right)^{\Dv_{Q,K}/\Iv_{Q,K}}$.
\end{proposition}
\begin{proof}
See \cite[Corollary 3.18]{KO7}.
\end{proof}

\subsection{Function fields and ramification}

\begin{proposition} \label{389}
Let $K$ be a function field over a perfect field $k$. Let $\overline{k}$ be an algebraic closure of $k$ in an algebraic closure of $K$. The following hold:
\begin{enumerate}
\item Let $M/K$ be a finite normal extension. Then only finitely many primes $P \in \mathcal{P}_{K/k}$ have an extension $Q \in \mathcal{P}_{M/k}$ with the property that $\Iv_{Q,K} \neq 0$.
\item For all primes $Q \in \mathcal{P}_{\overline{k}K/k}$ one has $\Iv_{Q,K}=0$. Furthermore, if $k$ is integrally closed in $K$, then a prime $P \in \mathcal{P}_{K/k}$ has a unique extension to $\mathcal{P}_{\overline{k}K/k}$ if and only if $P \in \mathcal{P}_{K/k}^1$.
\end{enumerate}
\end{proposition}
\begin{proof}
i. Since valuation extend unique in purely inseparable field extensions, we may assume that $M/K$ is Galois. In that case, the statement is well-known: only finitely many primes ramify.

ii. This follows from \cite[Theorem 3.6.3]{ST}.

\end{proof}

\section{Proofs of the theorems}

\begin{definition}
Let $K$ be a function field over a field $k$ and let $P \in \mathcal{P}^1_{K/k}$. Let $\overline{K}$ be an algebraic closure of $K$. Let $M/K$ be a normal extension of $K$ inside $\overline{K}$ with group $G=\Aut_K(M)$. Let $\overline{k}$ be the algebraic closure of $k$ inside $\overline{K}$. Set $\Gamma=\Aut_k(\overline{k})$. Let $Q$ be an extension of $P$ to $M$. A \emph{Frobenius} for $Q/P$ is a continuous morphism $\Gamma \to G$ such that
that the diagram, where $\psi$ is the natural map,
\[
\xymatrix{
\Gamma \ar[r]^{\varphi} \ar[rd]_{\psi} & G \ar[d]^{\pi} \\
 & \Aut(\overline{k} \cap M/k)\cong \Aut((M \cap (\overline{k}K))/K)=G/N
}  
\]
commutes and such that
$\im(\varphi) \Iv_{Q,K}=\Dv_{Q,K}$. The set of all such Frobenius maps is denoted by $\Frob(Q/P)$. 
\end{definition}

For the next proposition and lemma assume that $M/K$ is a normal extension of function fields over a perfect field $k$ inside $\overline{K}$. Let $G=\Aut_K(M)$. Let $P \in \mathcal{P}_{K/k}^1$ with $Q$ above it in $M$.

Suppose $\varphi \in \Frob(Q/P)$. We use the notation as in the definition of a Frobenius. Consider $\Graph(\varphi) \subseteq \Gamma \times_{G/N} G= \Gal(\overline{k}M/K)$, which is a closed subgroup. Set $M^{\varphi}=(\overline{k}M )^{\Graph(\varphi)}$. Note that $M^{\varphi}/K$ is a finite extension if $M/K$ is finite. Furthermore, $M^{\varphi}$ is geometrically irreducible over $k$ since $k$ is perfect. Finally, we have $\overline{k} M^{\varphi}=\overline{k}M$. 

\begin{proposition} \label{25c8888}
There is $\varphi \in \Frob(Q/P)$ such that $\mathcal{P}^1_{M^{\varphi}/k}$ is not empty. 
\end{proposition}
\begin{proof}
Let $\overline{Q}$ be an extension of $Q$ to $\overline{k}M$. We have the following split exact sequence (Theorem \ref{984}ii):
\begin{align*}
0 \to \Iv_{\overline{Q},K} \to \Dv_{\overline{Q},K} \to \Gamma \to 0.
\end{align*}
Let $\varphi_0 \colon \Gamma \to \Dv_{\overline{Q},K} \subseteq \Gamma \times_{G/N} G$ be such a splitting. Note that $\im(\varphi_0)$ is the graph of a function $\varphi \colon \Gamma \to \Dv_{Q,K}$. One has $\im(\varphi_0) \Iv_{\overline{Q},K} = \Dv_{\overline{Q},K}$. This gives $\im(\varphi) \Iv_{Q,K}=\Dv_{Q,K}$ as required (Theorem \ref{984}iv). The commutativity of the diagram follows since $\Graph(\varphi)=\im(\varphi_0) \subseteq \Gamma \times_{G/N} G$.

We will show $Q'=\overline{Q}|_{M^{\varphi}} \in \mathcal{P}^1_{M^{\varphi}/k}$. By construction we have $\Dv_{\overline{Q},K} \supseteq \Graph(\varphi)$. Hence we obtain $\Dv_{\overline{Q},M^{\varphi}} = \Dv_{\overline{Q},K} \cap \Gal(M/M^{\varphi}) = \Gal(M/M^{\varphi})$ (Theorem \ref{984}iii). Hence $\overline{Q}$ is the unique prime above $Q'$ in $\overline{k}M$ (Theorem \ref{984}i) and $Q'$ is rational (Proposition \ref{389}ii).
\end{proof}

\begin{lemma} \label{25c9999}
Let $\varphi \in \Frob(Q/P)$. Consider the natural map $\psi \colon \mathcal{P}^1_{M^{\varphi}/k} \to \mathcal{P}^1_{K/k}$. Let $P \in \im(\psi)$ with the property that any extension $Q$ of $P$ to $M$ satisfies $\Iv_{Q,K}=0$. Then there is a prime $Q$ of $M$ above $P$ with $\Dv_{Q,K}=\im(\varphi)$.
\end{lemma}
\begin{proof}
Let $Q''$ be a valuation on $\overline{k}M$ extending $P$ such that $Q'=Q''|_{M^{\varphi}} \in \mathcal{P}^1_{M^{\varphi}/k}$ (Theorem \ref{984}i). Note that $\Iv_{Q'',K}=0$ (this follows from Proposition \ref{389}). Set $Q=Q''|_M$ and $P'=Q''_{\overline{k} K }$. 
As $Q'$ is rational, the natural injective map
\begin{align*}
\Dv_{Q'',K} \cap \Graph(\varphi)=\Dv_{Q'',M^{\varphi}} \cong \Aut(k_{Q''}/k_{Q'}) \to \Dv_{Q'',K} \cong \Aut(k_{Q''}/k)
\end{align*}
is surjective (Theorem \ref{984}iv). Hence we find $\Dv_{Q'',K} \subseteq \Graph(\varphi)$.
The map $\Dv_{Q'',K} \to \Dv_{P',K}=\Gamma$ is surjective (Theorem \ref{984}iv, \cite[Theorem 3.6.3]{ST}). As $\Graph(\varphi)$ is a graph, this shows that $\Dv_{Q'',K}=\Graph(\varphi)$. We deduce $\Dv_{Q,K}=\im(\varphi)$ (Theorem \ref{984}iv).
\end{proof}

\begin{proof}[Proof of Theorem \ref{25c12}]
With the help of Theorem \ref{2c411} we see that equivalently we need to prove the following. Let $L/K$ be a finite extension of function fields over $k$. Assume that the induced map $f_k \colon \mathcal{P}^1_{L/k} \to \mathcal{P}^1_{K/k}$ is not surjective. Show that $\mathcal{P}^1_{K/k} \setminus \im(f_k)|=|k|$.

Let $M$ be a finite normal extension of $K$ such that $X=\Hom_K(L,M) \neq \emptyset$. Assume $P \in \mathcal{P}^1_{K/k}$, $P \not \in \im(f_k)$. Let $Q$ be an extension of $P$ to $M$. Let $\varphi \in \Frob(Q/P)$ with $\mathcal{P}^1_{M^{\varphi}/k} \neq \emptyset$ (Proposition \ref{25c8888}). Since $k$ is a large field, one has $|\mathcal{P}^1_{M^{\varphi}/k}|=|k|$ (\cite[Proposition 5.4.3]{JAR2}). Note that we have $\im(\varphi) \Iv_{Q,K}=\Dv_{Q,K}$. As $P \not \in \im(f_k)$ we conclude from Corollary \ref{1c388} that $\left(\Iv_{Q,K} \backslash X \right)^{\Dv_{Q,K}/\Iv_{Q,K}} = \emptyset$. This implies $X^{\im(\varphi)}=\emptyset$. Consider the map $\psi \colon \mathcal{P}^1_{M^{\varphi}/k} \to \mathcal{P}^1_{K/k}$. Let $T$ be the set of $P' \in \im(\psi)$ such that for any extension $Q'$ of $P'$ to $M$ one has $\Iv_{Q',K}=0$. Then $T$ has cardinality $|k|$ as well (Propostion \ref{389}i, Theorem \ref{984}i). For $P' \in T$ there is a prime $Q'$ of $M$ above $P'$ with $\im(\varphi)=\Dv_{Q',K}$ (Lemma \ref{25c9999}). We find 
\begin{align*}
\left(\Iv_{Q',K} \backslash X \right)^{\Dv_{Q',K}/\Iv_{Q',K}}=X^{\Dv_{Q',K}}=X^{\im(\varphi)}=\emptyset.
\end{align*}
From Corollary \ref{1c388} we conclude that $P' \not \in \im(f_k)$. This finishes the proof.
\end{proof}

\begin{proof}[Proof of Corollary \ref{25c14}]
i: The statement is true if $f$ is constant (\cite[Proposition 5.4.3]{JAR2}). Assume $f \in k(x)$ is not constant. Then $f$ induces a finite morphism $\Ps^1_k \to \Ps^1_k$ and we apply Theorem \ref{25c12}.

ii: This follows from i, since the map sends the point at infinity to the point at infinity.
\end{proof}

\end{document}